\documentclass[12pt]{article}

\usepackage{graphicx}
\usepackage{amsmath,amsthm,amssymb,color}
\usepackage[noadjust,sort]{cite}

\date{}

\normalsize

\numberwithin{equation}{section}

\def\({\bigl(}
\def\){\bigr)}

\newtheorem{thm}{Theorem}[section]
\newtheorem{cor}{Corollary}[section]
\newtheorem{lemma}{Lemma}[section]

\def\abs#1{\lvert#1\rvert} 
\def\Abs#1{\bigl\lvert#1\bigr\rvert} 
\def\dfrac#1#2{\lower0.15ex\hbox{\large$\textstyle\frac{#1}{#2}$}}
\def\({\bigl(}
\def\){\bigr)}
\def\st{\mathrel{|}}

\let\eps=\varepsilon

\def\calZ{\mathcal{Z}}

\def\E{\operatorname{\mathbb{E}}}

\def\Diam{\operatorname{diam}}
\def\Reals{{\mathbb{R}}}
\def\Complexes{{\mathbb{C}}}
\def\Mix{\operatorname{Mix}}

\def\nicebreak{\vskip 0pt plus 50pt\penalty-300\vskip 0pt plus -50pt }

\begin{document}

\title{On a bound of Hoeffding in the complex case}

\author{Mikhail Isaev${}^*{}^\dag$~~and~~Brendan~D.~McKay\vrule width0pt height2ex\thanks
 {Research supported by the Australian Research Council.}\\
\small ${}^*$Research School of Computer Science\\[-0.9ex]
\small Australian National University\\[-0.9ex]
\small Canberra ACT 0200, Australia\\[0.3ex]
\small ${}^\dag$Moscow Institute of Physics and Technology\\[-0.9ex]
\small Dolgoprudny, 141700, Russia\\[-0.3ex]
\small\texttt{mikhail.isaev@anu.edu.au, bdm@cs.anu.edu.au}
}

\maketitle

\begin{abstract}
 It was proved by Hoeffding in 1963 that a real random variable
 $X$ confined to $[a,b]$ satisfies
 $\E e^{X-\E X} \le e^{(b-a)^2/8}$.  We generalise this to 
 complex random variables.
\end{abstract}

\section{Introduction}\label{intro}

A celebrated concentration inequality of Hoeffding relies on the
following bound.
\begin{lemma}[{\cite{Hoeffding}}]\label{Hoeff}
  Let $X$ be a real random variable such that $a\le X\le b$.  Then
  \[
      1\le \E e^{X-\E X} \le e^{(b-a)^2/8}.
  \]
\end{lemma}

Given the hundreds of references to this inequality in the literature,
we believe that a similarly tight bound for complex random
variables may also have application.  Indeed, our ongoing investigation
of complex martingales related to multidimensional asymptotics encounters
such a need.  Our aim, therefore, is to find a complex analogue
of Lemma~\ref{Hoeff}.

There are several possible complex replacements for the real
bounds $a\le X\le b$.
A~natural choice would be to confine
$Z$ to a disk of given radius, but we will use the weaker condition
that the support of $Z$ has bounded diameter.
This measure of spread naturally arises in the
study of separately Lipschitz functions~\cite{Fan}, also called 
functions satisfying the
bounded difference condition~\cite{Boucheron}.
For a complex random variable $Z$, define the diameter of~$Z$ to be
\begin{align}
    \Diam Z &= \inf\, \bigl\{ c\in\Reals \st P(\abs{Z_1-Z_2} > c) = 0 \bigr\}, \label{diam1}\\
     &\kern3em\text{where $Z_1,Z_2$ are independent copies of~$Z$,} \notag
\end{align}
with the infimum of an empty set taken to be~$\infty$.

Since Hoeffding's bound states that 
$\E e^{X-\E X}$ is concentrated near~1, the natural complex
analogue is to bound the distance of $\E e^{Z-\E Z}$ from~1.
This is the nature of our main theorem.

\begin{thm}\label{main}
  Let $Z$ be a complex random variable with $\Diam Z\le d$.  Then
  \[ \abs{\E e^{Z-\E Z}-1} \le e^{d^2/8}-1. \]
\end{thm}

Since $\abs {Z-\E Z}\le \alpha$ implies $\Diam Z\le 2\alpha$, Theorem~\ref{main}
has a simple consequence, which is a complex version of an inequality
used many times in proving Azuma-type inequalities.
\begin{cor}\label{simplecor}
 Let $Z$ be a complex random variable with $\abs{Z-\E Z}\le \alpha$ a.s..
 Then \[ \abs{\E e^{Z-\E Z}-1} \le e^{\alpha^2/2}-1. \]
\end{cor}
Note that in both the theorem and its corollary, distrubutions
supported on $\{-1,+1\}$ are enough to show that the constants
(respectively $\tfrac18$ and $\tfrac12$) cannot be reduced.

Corollary~\ref{simplecor} and weaker versions of Theorem~\ref{main}
can be proved by simpler means.
For example, the referee noted that Corollary~\ref{simplecor}
(which implies Theorem~\ref{main} with constant~$\tfrac12$
instead of  $\tfrac18$) can be proved by representing complex numbers
by real matrices and applying result of Tropp~\cite[Lemmas 7.6--7]{Tropp}.

Hoeffding actually found the best possible bound on $e^{X-\E X}$ in the
real case, as we recall in the next section.  
In Section~\ref{ss:small} we show that for $d\le 3.12$
the same tighter bound holds in the complex case too, making use
of a lemma that only random variables with support of at most three
points need to be considered.  Then in Section~\ref{ss:main} we complete
the proof of Theorem~\ref{main} for all~$d$.

\section{Results}

Hoeffding's paper~\cite{Hoeffding} used the convexity of the exponential
function to find the tightest possible bound in the real case.
For $d>0$,
consider the random variable $X_d$ supported on $\{0,d\}$
with 
\[
    P(X_d=d) = \frac{e^{d}-1-d}{d(e^{d}-1)}
       = \dfrac12 - \dfrac1{12d} + O(d^3).
\]
We find that $\E e^{X_d-\E X_d}-1 = G(d)$, where
\begin{align*}
    G(d) &= \frac{
       \exp\Bigl(-\frac{e^d-1-d}{e^d-1}\Bigr)
       - 2\exp\Bigl(\frac{d e^d-e^d+1}{e^d-1}\Bigr)
       + \exp\Bigl(\frac{2d e^d-e^d-d+1}{e^d-1}\Bigr)}
       {d(e^d-1)} - 1\\
       &= \dfrac{1}{8}d^2 + \dfrac{7}{1152}d^4 + O(d^6)
       = \exp\( \dfrac{1}{8}d^2 - \dfrac{1}{576}d^4 + O(d^6)\) - 1.
\end{align*}
We can now state Hoeffding's bound in its strongest form.
\begin{lemma}\label{Hoeff3}
 Let $X$ be a real random variable such that $a\le X\le b$.  Then
  \[
      \abs{\E e^{X-\E X} -1}\le G(b-a) \le e^{(b-a)^2/8}-1,
  \]
  where the first inequality holds with equality if and only if
  $X=X_{b-a}+a$ almost surely.
\end{lemma}

\subsection{The complex case: tight bound for small diameter}\label{ss:small}

Let $Z_1,\ldots,Z_n$ be complex random variables and 
let $c_1,\ldots,c_n$ be nonnegative real numbers with
$c_1+\cdots+c_n=1$.
Define the \textit{mixture} $Z=\Mix_{c_1,\ldots,c_n}(Z_1,\ldots,Z_n)$ by
\[
   P(Z\in A) = \sum_{k=1}^n c_k\,P(Z_k\in A)
\]
for every measurable set~$A\subseteq\Complexes$.  A standard
property of mixtures is that
 \begin{equation}\label{expmix}
     \E F\(\Mix_{c_1,\ldots,c_n} (Z_1,\ldots,Z_n)\)
     = \sum_{j=1}^n c_j \E F(Z_j)
 \end{equation}
 for any measurable function $F:\Complexes\to\Complexes$
 for which the expectations exist.

Let $\calZ_d$ be the class of all complex random variables $Z$ with
$\E Z=0$ and $\Diam Z\le d$, and let $\calZ_d^{(k)}$ be the subclass
of $\calZ_d$ consisting of those variables supported on at most $k$~points.

\begin{lemma}\label{max}
Let $F:\{ z \in\Complexes \st \abs z \le d\} \to\Complexes$ be continuous.
Then
\[
    \sup_{Z\in\calZ_d}\, \abs{\E F(Z)} =  \sup_{Z\in\calZ_d^{(3)}} \abs{\E F(Z)}.
\]
\end{lemma}
\begin{proof}
 This is an example of the ``Carath\'eodory Principle'', see for example
  \cite{Hoeffding0,MR,Pinelis}.
  Since we didn't find a statement in the
  literature that exactly matches our needs, we outline the proof.
  
  First, by a simple induction, any $Z\in\calZ_d$ with finite support can
  be written as a mixture of members of $\calZ_d^{(3)}$.  This is true
  because, for any finite set of points in $\Complexes$ having the origin
  in its convex hull, there is a subset of three or fewer points having the
  origin in its convex hull.
  By~\eqref{expmix}, this implies that the lemma holds when $Z$ has
  finite support.
  
  For more arbitrary $Z\in\calZ_d$, we can use the continuity of~$F$
  to show that for any $\eps>0$ there is $Z'\in\calZ_d$ of finite support
  such that $\E Z'=0$ and $\abs{\E F(Z) - \E F(Z')} \le \eps$.
  Allowing $\eps$ to tend to~0 completes the proof.
\end{proof}

We now return to Hoeffding's bound using
Lemma~\ref{max} for $F(z)=e^z-1$.
Obviously $\abs{\E e^{Z-\E Z}-1}=0$ if
$Z$ is constant, so we need to consider the cases of
2-point and 3-point supports.
Since a random variable supported on 3 collinear points is a mixture
of two random variables supported on 2 points, in the case of~3~points
only the non-collinear case needs to be considered.

\begin{lemma}\label{twopoints}
  If $Z\in\calZ_d^{(2)}$, then $\abs{\E e^Z-1}\le G(d)$.
\end{lemma}
\begin{proof}
  The case of real $Z$ is treated in Lemma~\ref{Hoeff3}.
  More generally, since $\E Z=0$,
  $Z=e^{i\theta}X$ where $X$ is real.
  Since $\E X=0$, there are $x,x'\ge 0$ such that $X$
  has support $-x$ with probability $x'/(x+x')$ and
  $x'$ with probability $x/(x+x')$.
  For any odd $k$ we can calculate that
\[
    \E X^k = \frac{x^k}{1+\rho} (\rho^k-\rho),
    \text{~ where $\rho=x'/x$.}
\]
This shows that either $X$ or $-X$ has only nonnegative moments.
By adding $\pi$ to $\theta$ if necessary, we assume that the former holds.
Now, recalling that $\E Z=\E X=0$, we can calculate
\begin{align*}
  \abs{\E e^Z-1} &= \Bigl|\, \sum_{k=2}^\infty \dfrac{1}{k!}\E Z^k \Bigr|
                          = \Bigl|\, \sum_{k=2}^\infty \dfrac{1}{k!}e^{ik\theta} \E X^k \Bigr| \\
                          &{}\le \sum_{k=2}^\infty \dfrac{1}{k!}\E X^k = \E e^X-1,
\end{align*}
which implies that $\abs{\E e^Z-1}\le G(d)$ by the real case.
\end{proof}

The case of a support of three points is considerably more difficult.
To begin, define
\[
   d_0 = \sup \, \{ d\in\Reals \st \Re \E e^Z\ge 0\text{~for all~}Z\in\calZ_d \}.
\]
Since $\Re e^z\ge 0$ for $\abs{z}\le \pi/2$ we see that $d_0\ge\pi/2$.
The actual value is almost twice as large.

\begin{lemma}\label{real0}
$d_0\approx 3.120491233$.
\end{lemma}
\begin{proof}
  By applying  Lemma~\ref{max} to $F(z)=e^d-\Re z$, we see that
  only supports of two points or three non-collinear points
  need to be considered.
  
  For supports of two points, we have by definition
  \[
      \inf \{ \Re\E e^Z\st Z\in\calZ_d^{(2)} \}
      = \inf_{\ell\in[0,d], x\in[0,1], \theta\in[-\pi,\pi]}
         \Re \( xe^{\ell(1-x)e^{i\theta}} + (1-x)e^{-\ell xe^{i\theta}} \).
  \]
There appears to be no closed form for the infimum.
  However, careful numerical
  computation shows that the infimum crosses~0 at
  $d=d_2\approx 3.120491233$, which occurs when
  $\ell=d$, $x\approx 0.636527202$ and $\theta\approx 1.9198934984$.
  
  If $d_0<d_2$, there is some $d<d_0$ such that 
  \[
     0 < \inf \{ \Re\E e^Z\st Z\in\calZ_d^{(3)}\} 
        <  \inf \{ \Re\E e^Z\st Z\in\calZ_d^{(2)}\}.
  \]   
  By the compactness of $\calZ_d^{(3)}$, the first infimum
  is realised by some $Z\in \calZ_d^{(3)}$ with a support
  of three non-collinear points $\{z_1,z_2,z_3\}$.  Define
  $Z_{(x,y)}$ to be the random variable with the same support, but
   mean at $x+iy$ (thus $Z=Z_{0,0}$).
   As is well known, $x,y$ determine the probabilities
   at $z_1,z_2,z_3$ linearly, so for some complex constants $A,B,C$
 we have
 \begin{align}
     \E e^{Z_{(x,y)}-\E Z_{(x,y)}}-1 &= e^{-x-iy}(Ax + By + C)-1 \notag\\
         &= C-1 + (A-C)x + (B-iC)y + O(x^2+y^2), \label{expan}
 \end{align}
 valid whenever $x+iy$ lies in the convex hull of $\{z_1,z_2,z_3\}$.
 
 In order for $Z$ to be a local minimum for the real part, the coefficients
 of ~$x$ and~$y$ in~\eqref{expan} must be purely imaginary.  Therefore,
 for some $v,w\in\Reals$, $A=C+iv$ and $B=iC+iw$.  Substituting in these
 values and writing $C=c_0+ic_1$ we find
  \[
    \E \Re{e^{Z_{(x,y)}-(x+iy)}} = c_0
       + [x,y] R [x,y]^T + O(\abs x^3+\abs y^3),
 \]
 where
 \[
     R = \left[ \begin{matrix}
                -\tfrac12 c_0  & \tfrac12 (c_1+v) \\[0.5ex]
               \tfrac12 (c_1+v) & \tfrac12 c_0 +w
            \end{matrix}\right].
 \]
 The smallest eigenvalue of $R$ is
 $\tfrac12 w - \tfrac12 \sqrt{(w+c_0)^2 + (v + c_1)^2}$,
 which is clearly negative for $c_0>0$.  Thus $(x,y)=(0,0)$ is not a 
 local minimum for $\E\Re\,{e^{Z_{(x,y)}-\E Z_{(x,y)}}}$, contrary to our
 assumption.  This proves that $d_0=d_2$.
\end{proof}

\begin{thm}\label{smalld}
 If $\Diam Z=d\le d_0$, then $\abs{\E e^Z-1}\le G(d)$.
\end{thm}
\begin{proof}
 We can rely on
 continuity to assume that $d < d_0$.
 
 By Lemma~\ref{max} with $F(z)=e^z-1$, we have
$\sup_{Z\in\calZ_d}\, \abs{\E e^Z-1} =  \sup_{Z\in\calZ_d^{(3)}} \abs{\E e^Z-1}$,
and by compactness the supremum is achieved for some $Z\in\calZ_d^{(3)}$.
We first show that no such local maximum occurs when the support
consists of 3 non-collinear points, implying that it occurs in~$\calZ_d^{(2)}$.
By Lemma~\ref{twopoints}, this will complete the proof.

So, suppose that the maximum value of $\abs{\E e^Z-1}$ for
$Z\in\calZ_d^{(3)}$ where the support of~$Z$ consists of three
non-collinear points.
 Let $Z_{(x,y)}$ be the random variable with the same support
 as~$Z$, but mean at $x+iy$.  As in the proof of Lemma~\ref{real0},
 \eqref{expan} holds.
 Since we are assuming this to be a local maximum
 for $\abs{\E e^{Z}-1}$, $A-C$ and
 $B-iC$ must be orthogonal to $C-1$, i.e. real multiples of $i(C-1)$.
 That is, for some $v,w\in\Reals$, $A=C+iv(C-1)$ and
 $B=iC+iw(C-1)$.
 Substituting in these values, writing $C=c_0+i c_1$ for
 $c_0,c_1\in\Reals$, and defining $\varDelta=\abs{C-1}$ we can expand
 \[
    \Abs{\E\,e^{Z_{(x,y)}-\E Z_{(x,y)}}-1}^2 = \abs{C-1}^2
       + [x,y] Q [x,y]^T + O(\abs x^3+\abs y^3),
 \]
 where
 \[
     Q = \left[ \begin{matrix}
               \varDelta v^2 + c_0 - c_1^2 - c_0^2 
                       & \varDelta vw - c_1 \\[0.5ex]
               \varDelta vw - c_1 & \varDelta w^2 + c_0 - 1
            \end{matrix}\right].
 \]
 We now show that $Q$ cannot be negative semidefinite.
 If $\varDelta=0$ then $\E e^Z-1=0$, which is clearly not a maximum,
 so assume $\varDelta>0$.  
 The trace of $Q$ is $\varDelta(v^2+w^2-1)$, which is impossible for
 a negative semidefinite matrix if $v^2+w^2>1$, so assume
 $v^2+w^2\le 1$.  The determinant of $Q$ is
 $\varDelta\( -w^2c_0^2 - (1-v^2-w^2)c_0 - (v-wc_1)^2\)$, which
 is negative (since $d<d_0$ implies $c_0>0$), which is also
 impossible for a negative semidefinite matrix.
 Therefore, there is no local maximum here and the proof is complete.
\end{proof}

We have no reason to believe that Theorem~\ref{smalld} requires
the condition $d\le d_0$, and expect that it is true for all~$d$.
However, the same proof is insufficient since it is possible for
local maxima to occur for supports of three points.  However,
we can now complete the proof of Theorem~\ref{main} for all~$d$.

\subsection{Proof of Theorem~\ref{main}}\label{ss:main}

We need two technical bounds whose uninteresting proofs are omitted,
and a standard result on planar sets.

\begin{lemma}\label{technical}
For $t\ge 3$ we have
\begin{align}
    G(t)+1 &\le 0.9\, e^{t^2/8}\,. \label{tech2} \\
     \sqrt{G(2t)+1} &\le 1.65\, e^{t^2/8}\, \label{tech1}
\end{align}
\end{lemma}

\begin{lemma}\label{diamprops}
Let $Z$ be a bounded complex random variable.  Then
$Z$ is almost surely confined to some closed disk of
radius $\frac{1}{\sqrt 3}\Diam Z$.
\end{lemma}
\begin{proof}
This follows from a standard result on convex sets, see \cite[Thm.~12.3]{Lay}
for example.  An equilateral triangle shows that the constant cannot
be reduced.
\end{proof}

\begin{proof}[Proof of Theorem~\ref{main}]
  For $d\le 3$, the theorem follows from Theorem~\ref{smalld},
  so we can assume that $d\ge 3$.
    
  By Lemma~\ref{diamprops}, there is some $a\in\Complexes$ such
  that $\abs{Z-a}\le \tfrac{1}{\sqrt3}\,d$.
  We will find two bounds on $\abs{\E Ze^Z}$.
  First we argue that
  \begin{align}
     \abs{\E Ze^Z} &\le \abs{\E (Z-a)e^Z} + \abs{a\E e^Z} \notag\\
        &\le \dfrac{1}{\sqrt3}\,d\E\,\abs{e^Z} + \abs{a}\,\abs{\E e^Z} \notag\\
        &\le \dfrac{1}{\sqrt3}\, 0.9 \,d e^{d^2/8} + \abs{a}\,\abs{\E e^Z},
        \label{Ebnd1}
  \end{align}
  where we have used Lemma~\ref{Hoeff3} to bound
  $\E\abs{e^Z}=\E e^{\Re Z}$, and~\eqref{tech2}.
  On the other hand
  \begin{align}
     \abs{\E Ze^Z} &= \abs{\E Z(e^Z-\E e^Z)}
               \le \sqrt{\E \abs{Z}^2\,\E\abs{e^Z-\E e^Z}^2} \notag \\
       &\le \sqrt{\E\abs{Z-a}^2 - \abs{a}^2} \;
                \sqrt{\E\abs{e^{2Z}}-\abs{\E e^Z}^2} \notag \\
       &\le \sqrt{d^2/3 - \abs{a}^2} \;
           \sqrt{\vphantom{e^d}\smash{1.65^2 e^{d^2/4}-\abs{\E e^Z}^2}},
             \label{Ebnd2}
  \end{align}
  where we used Lemma~\ref{Hoeff3} to bound $\E e^{2\Re Z}$, 
  and~\eqref{tech1}.
  
  Taking the average of~\eqref{Ebnd1} and~\eqref{Ebnd2}, we have
  \begin{align}
     \abs{\E Ze^Z} &\le \dfrac{1}{\sqrt3}\,d e^{d^2/8}
        \( 0.45 + 0.825\alpha\beta + 0.825\sqrt{(1-\alpha^2)(1-\beta^2)}\,\) 
         \notag \\
      & \le \dfrac34 d e^{d^2/8}, \label{mbound}
  \end{align}
  where $\alpha,\beta$ are defined by 
  $\abs{a}=\tfrac{1}{\sqrt3}\,\alpha d$ and $\abs{\E e^Z}=1.65\,\beta e^{d^2/8}$,
  and we have used that $xy+\sqrt{(1-x^2)(1-y^2)}\le 1$ for $0\le x,y\le 1$.
  
  We can now complete the proof.  We have
  \[
     \E e^Z-1 = \E e^{3Z/d}-1
       + \int_{3/d}^1 \E Z e^{sZ}\,ds . \\
  \]
  Since $\Diam(3Z/d)\le 3$, we have by Theorem~\ref{smalld} that 
  $\abs{\E e^{3Z/d}-1}\le e^{3^2/8}-1$, and since we are assuming
  that $d\ge 3$, \eqref{mbound} gives that
  $\abs{\E Z e^Z}\le \frac14 d^2s\, e^{d^2s^2/8}$ for $s\ge 3/d$.
  Therefore,
  \[
    \abs{\E e^Z-1} \le e^{3^2/8}-1 
      + \int_{3/d}^1 \dfrac14 d^2s e^{d^2s^2/8}
    = e^{d^2/8}-1. \qedhere
  \]
\end{proof}

\begin{figure}[ht]
\centering
\unitlength=1cm
\begin{picture}(12,7.7)(0,0)
  \put(1,0){\includegraphics[scale=0.4]{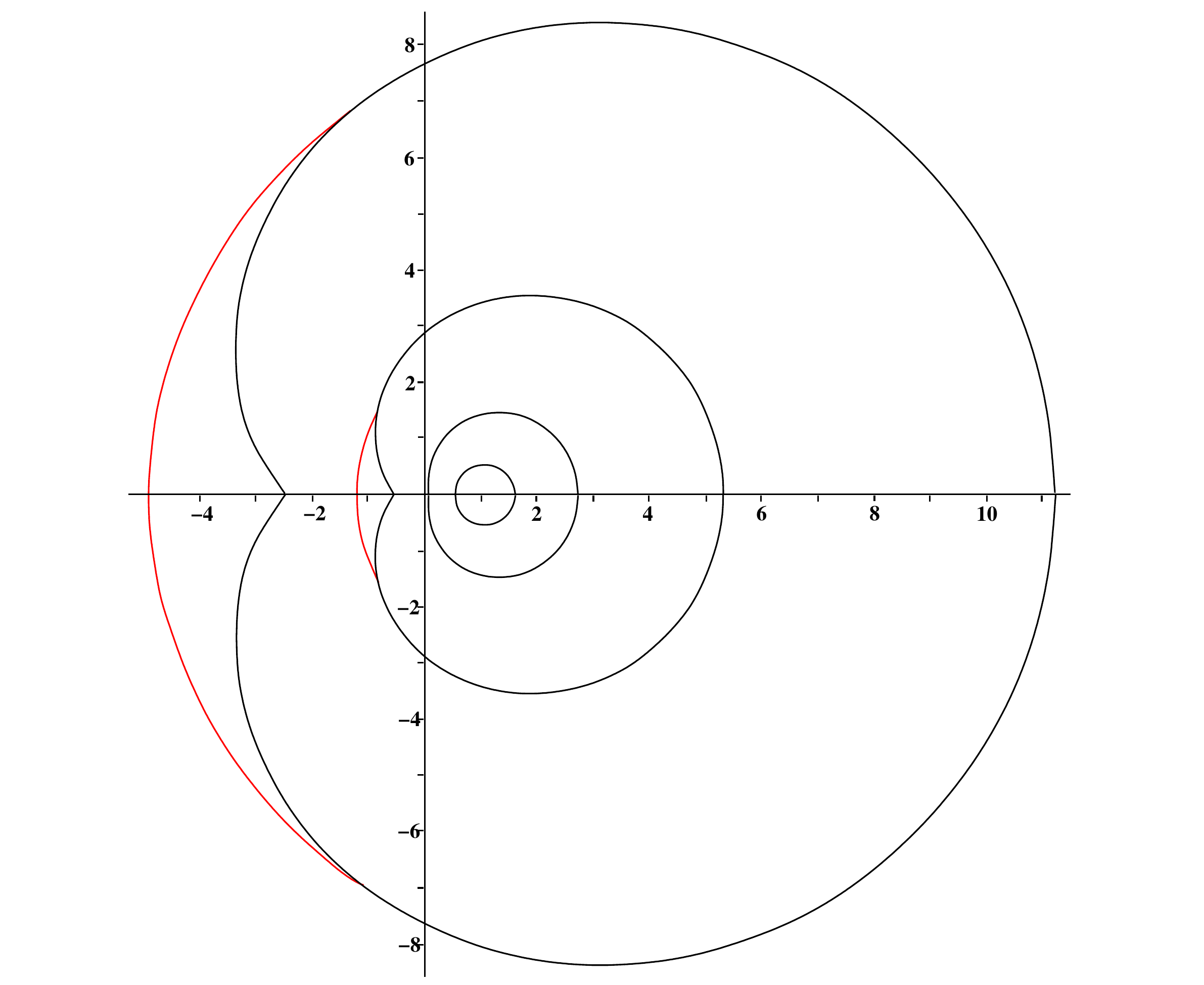}}
  \put(5.1,4.20){\small $d{=}2$}
  \put(5,4.75){\small $d{=}3$}
  \put(6,5.5){\small $d{=}4$}
  \put(8.2,7){\small $d{=}5$}
\end{picture}
\caption{Boundaries of $S_d^{(2)}$ (black) and
$S_d^{(3)}$ (red) for $d=2,3,4,5$}
\end{figure}
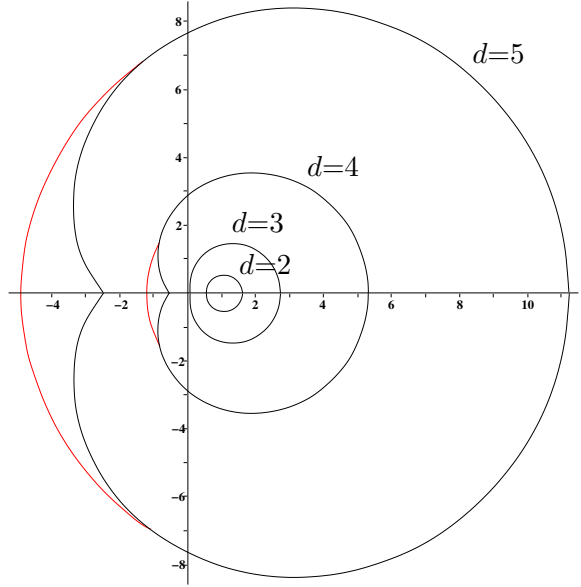

\subsection{Conclusions}

In conclusion, we note some questions that we have not
answered.  Our prediction
that Theorem~\ref{smalld} holds for all~$d$ is one of them.
We can also ask 
for an exact description of the boundary of possible values of 
$\E e^{Z-\E Z}$ when $\Diam Z\le d$. 
Define
$S_d = \{ \E e^Z \st Z\in \calZ_d \}$
and $S_d^{(k)} = \{ \E e^Z \st Z\in \calZ_d^{(k)} \}$.

Figure~1 shows $S_d^{(2)}$ and $S_d^{(3)}$ for $d=2,3,4,5$.
Note that manifestly $S_d^{(2)}$ and $S_d^{(3)}$ are not
always equal.  However it could be that the parts of them
in the right half-plane are equal.

By considering mixture with an identically-zero random variable,
we find that the region $S_d$ is star-like from the point~1, and therefore
simply connected.
Lemma~\ref{max} applied to functions of the form
$F(z)=\Re\, e^{z-i\theta}$ shows
that $S_d$ and $S_d^{(3)}$ have the same convex hull,
and therefore are the same if $S_d^{(3)}$ is convex.
However,  $S_d^{(3)}$ is not always convex; there is a
shallow indentation on the left side when $d=3$, for example.
There is also an indentation where the red and black curves
in the figure meet.

\nicebreak

\end{document}